\documentclass{article}

\usepackage{amsmath, mathrsfs, amssymb,amsthm,hyperref,tikz,tikz-3dplot}
\hypersetup{pdfstartview={XYZ null null 1.25}}
\usepackage[all]{xy}

\theoremstyle{plain}

\newtheorem{theorem}{Theorem}[section]{\bfseries}{\itshape}
\newtheorem{proposition}[theorem]{Proposition}{\bfseries}{\itshape}
\newtheorem{definition}[theorem]{Definition}{\bfseries}{\upshape}
\newtheorem{lemma}[theorem]{Lemma}{\bfseries}{\upshape}
{\bfseries}{\upshape}
\newtheorem{corollary}[theorem]{Corollary}{\bfseries}{\upshape}
{\bfseries}{\upshape}
{\bfseries}{\upshape}
{\bfseries}{\upshape}

\newcommand{\bw}{\bigwedge}
\newcommand{\bv}{\bigvee}
\newcommand{\UPPi}{\prod_U P_i}

\title{No finite axiomatizations for posets embeddable into distributive lattices}
\author{Rob Egrot}
\date{}

\begin{document}
\maketitle
\begin{abstract}
Let $m$ and $n$ be cardinals with $3\leq m,n\leq\omega$. We show that the class of posets that can be embedded into a distributive lattice via a map preserving all existing meets and joins with cardinalities strictly less than $m$ and $n$ respectively cannot be finitely axiomatized.    
\end{abstract}

\section{Introduction}
Let $m$ and $n$ be cardinals with $3\leq m,n\leq\omega$. It is shown in \cite{VanA16} that the problem of deciding whether a given finite poset can be embedded into a distributive lattice via a map preserving existing meets and joins with cardinalities strictly less than $m$ and $n$ respectively is $\mathbf{NP}$-complete for all $m$ and $n$ except, possibly, the case where both $m$ and $n$ are equal to 3. By \cite[Proposition 3.1]{Var95}, polynomial time algorithms exist for checking whether a fixed first-order sentence holds in finite models. So, if a class of posets with this kind of embedding property for some suitable $m$ and $n$ were finitely axiomatizable, it would imply that $\mathbf{P}=\mathbf{NP}$. Needless to say, this implication strongly suggests that none of these classes is finitely axiomatizable. However, intuitive finite first-order axiomatizations do exist for semilattices in similar situations \cite{Bal69,Sch72}.   

Assuming finiteness, or a suitable choice principle, this problem of embedding posets into distributive lattices is equivalent to the problem of embedding posets into powerset algebras via maps preserving meets and joins smaller than specified cardinals $m$ and $n$. Note that, since $m$ and $n$ are greater than $2$, such an embedding will automatically preserve any relative complements that exist in the poset. This has been studied in \cite{Egr17,Egr16} using the terminology $(m,n)$-\emph{representable} (see Definition \ref{D:rep}). In particular, it was shown that all the classes where $m,n\leq\omega$ are elementary \cite[Theorem 4.5]{Egr16}, though explicit axioms are not known. In the cases where either $m$ or $n$ is equal to $\omega$, the corresponding class is not finitely axiomatizable. This was shown directly in \cite{Egr16}, and also follows from the corresponding result for semilattices \cite{Kea97}. However, the cases where $m$ and $n$ are both finite were left open. 

Since for $3\leq m,n\leq\omega$ the classes of $(m,n)$-representable posets are all elementary, they will be finitely axiomatizable if and only if their complements are elementary. By \L o\'s' theorem, these complements will be elementary only if they are closed under ultraproducts.

For a poset $P$ and cardinals $\alpha$ and $\beta$, the existence of an $(\alpha,\beta)$-representation for $P$ is equivalent to a separation property generalizing the separation of distributive lattices by prime filters (the Prime Ideal Theorem for distributive lattices). In this note we use this property to construct a sequence of finite posets, all of which fail to be $(3,3)$-representable, and an ultraproduct of this sequence which is $(\omega,\omega)$-representable, thus proving that the class of $(m,n)$-representable posets cannot be finitely axiomatizable for any choice of $n,m\geq 3$. 

The classes of $(\alpha,\beta)$-representable posets, when $\alpha$ and/or $\beta$ are uncountable, and the classes where \emph{all} meets and/or joins must be preserved, are known to not be elementary at all, though in some cases they are pseudoelementary. See \cite[Figure 2]{Egr17} for a summary. 

In Section \ref{S:reps} we introduce the basic notation, definitions and results for representable posets (using the notation of \cite{Egr16}). Finally in Section \ref{S:result} we construct the required sequence of posets and prove the necessary results to support our main claim.

\section{Representable posets}\label{S:reps}
We begin with some notational conventions. Given a poset $P$ and a subset $S\subseteq P$ we define $S^\uparrow =\{p\in P:p\geq q$ for some $q\in S\}$. Given $p\in P$ we define $p^\uparrow=\{p\}^\uparrow$. Given a set $I$, an ultrafilter $U$ of $\wp(I)$, and posets $P_i$ for $i\in I$ we let $\UPPi$ be the ultraproduct with respect to $U$. For an element of $\UPPi$ we write, e.g. $[x]\in\UPPi$.

\begin{definition}[$(\alpha,\beta)$-representable]\label{D:rep}
Let $\alpha$ and $\beta$ be cardinals. We say a poset $P$ is $(\alpha,\beta)$-representable if there is a field of sets $F$, and a 1-1 map $h:P\to F$ such that:
\begin{enumerate}
\item Whenever $S$ is a subset of $P$ with $|S|<\alpha$, if $\bw S$ exists in $P$, then $h(\bw S)=\bigcap h[S]$.
\item Whenever $T$ is a subset of $P$ with $|T|<\beta$, if $\bv T$ exists in $P$ then $h(\bv T)=\bigcup h[T]$.
\end{enumerate}
If $\alpha=\beta$ we just write $\alpha$-\emph{representable}.
\end{definition}

\begin{definition}[$(\alpha,\beta)$-filter]
Let $\alpha$ and $\beta$ be cardinals, let $P$ be a poset, and let $\Gamma$ be an up-closed subset of $P$. We say $\Gamma$ is an $(\alpha,\beta)$-filter if:
\begin{enumerate}
\item Whenever $S\subseteq\Gamma$ and $|S|<\alpha$, if $\bw S$ exists, then $\bw S\in \Gamma$.
\item Whenever $T\subseteq P$ with $|T|<\beta$, if $\bv T$ exists and $\bv T\in \Gamma$, then $T\cap \Gamma\neq\emptyset$.
\end{enumerate}
I.e. $\Gamma$ is both $\alpha$-complete and $\beta$-prime. If $\alpha=\beta$ we just write $\alpha$-\emph{filter}.
\end{definition}

The following result relates $(\alpha,\beta)$-representability to separation by $(\alpha,\beta)$-filters. It appears explicitly in this form as \cite[Theorem 2.7]{Egr16}, but the idea of using this kind of separation property for representability-like results for ordered structures has been in the literature for over 50 years (see e.g. \cite[Theorem 4]{ChaHor62}). This also arises in pointless topology, as separation by completely prime filters is equivalent to a frame being \emph{spatial}, i.e. isomorphic to the open set lattice of some topological space (see e.g. \cite[Section 3]{John83}). 

\begin{theorem}\label{T:rep}
Let $\alpha$ and $\beta$ be cardinals, and let $P$ be a poset. Then $P$ is $(\alpha,\beta)$-representable if and only if, for all $p,q\in P$, if $p\not\leq q$, then there is an $(\alpha,\beta)$-filter $\Gamma\subset P$ with $p\in\Gamma$ and $q\notin \Gamma$.
\end{theorem}
Of course there is a dual result stated in terms of ideals rather than filters, and the details of this can also be found in \cite[Section 2]{Egr16}.

The next lemma shows how we can translate the existence of certain $(m,n)$-filters in coordinate posets into the existence of a certain $(m,n)$-filter in their ultraproduct. It will play an important role in proving $(m,n)$-representability for our ultraproduct. 
\begin{lemma}\label{L:filters}
Let $I$ be a set, and let $U$ be a non-principal ultrafilter of $\wp(I)$. Let $3\leq m,n\leq\omega$. For each $i\in I$ let $P_i$ be a poset, and let $[a],[b] \in\UPPi$. Let $u\in U$ and suppose that for all $i\in u$ there is an $(m,n)$-filter, $\Gamma_i$, of $P_i$ with $a(i)\in\Gamma_i$ and $b(i)\not\in\Gamma_i$. Then there is an $(m,n)$-filter, $\Gamma$, of $\UPPi$ with $[a]\in\Gamma$ and $[b]\notin\Gamma$. 
\end{lemma}
\begin{proof}
Let $\mathcal L$ be the standard language of posets extended by the single unary predicate symbol $G$. In every poset $P_i$ with $i\in u$ we interpret this predicate using
\begin{equation*}
P_i\models G(p)\iff p\in \Gamma_i
\end{equation*}  
Then, by the definition of ultraproducts, we have $\UPPi\models G([a])$ and $\UPPi\not\models G([b])$. For every $i\in u$ we have $\{p\in P_i: G(p)\} = \Gamma_i$, and so $\{p\in P_i: G(p)\}$ is thus an $(m,n)$-filter. So $P_i$ satisfies the set of first-order sentences ensuring $\{p\in P_i: G(p)\}$ is an $(m,n)$-filter for all $i\in u$. Thus, by \L o\'s' theorem, $\{[c]\in \UPPi : G([c])\}$ is also an $(m,n)$-filter, and, since $G([a])$ but not $G([b])$, we are done. 
\end{proof}

\section{Non-finite axiomatizability}\label{S:result}
We construct a sequence $(P_k:k=0,1,2,...)$ of finite posets. Each of these posets fails to be $3$-representable (and thus fails to be $(m,n)$-representable for all $m,n\geq 3$), but as $k$ increases the posets become, in a sense, closer to being 3-representable. We then show that an ultraproduct of these posets is $\omega$-representable (and so is $(m,n)$-representable for all $3\leq m,n\leq\omega$). This shows that, for all $3\leq m,n \leq \omega$, the complement of the class of $(m,n)$-representable posets is not elementary, and thus that the class of $(m,n)$-representable posets cannot be finitely axiomatized. 

In order to construct $P_k$ we first recursively define the sets $N_n$ for $n\in\omega$ by
\begin{itemize}
\item $N_0=\{a,b,c,d\}$.
\item Given $N_n$ we define $N_{n+1}=\{e_{xy}:x$ and $y$ are distinct elements of $N_n\}$. I.e. we get an element of $N_{n+1}$ for every distinct pair of elements in $N_n$.
\end{itemize}

Then for all $n\in\omega$ we define 
\begin{equation*}\widehat{N}_n=\bigcup_{x\in N_n}\{x',x''\}\end{equation*}

Given $k<\omega$ we define the carrier of $P_k$ to be
\begin{equation*}
\{p,q\}\cup\bigcup_{n=0}^k N_n \cup \bigcup_{n=0}^k \widehat{N}_n
\end{equation*}
We assume, of course, that elements labeled differently are distinct. We define the order on $P_k$ as follows:
\begin{enumerate}
\item $x< p$ for all $x\in N_0=\{a,b,c,d\}$.
\item $x< x'$ and $x< x''$ for all $x\in N_n$ and for all $n\leq k$.
\item For all $1\leq n\leq k$, if $x,y\in N_{n-1}$ and $e_{xy}$ is the corresponding element of $N_n$, we have $e_{xy}< x'$, $e_{xy}<x''$, $e_{xy}<y'$, and $e_{xy}< y''$.
\item For all $x\in N_k$ we have $q<x'$ and $q<x''$.
\item $x\leq x$ for all $x\in P_k$.
\item No other elements are comparable.
\end{enumerate}

Figures \ref{F:P0}-\ref{F:Pk} illustrate the posets $P_0$, $P_1$ and $P_k$. We now prove some facts about $P_k$, from which we deduce that $P_k$ is indeed a poset for all $k\in\omega$, and also that it has certain features that will be useful to us.
\begin{lemma}\label{L:height}
If $x,y,z \in P_k$, and $x\leq y$ and $y\leq z$, then, either $x=y$, or $y=z$.
\end{lemma}
\begin{proof}
If $x\leq y$, then either: (1) $x\in N_0$ and $y=p$, (2) $x\in N_n$ and $y\in \widehat{N}_n$, (3) $x\in N_n$ and $y\in \widehat{N}_{n-1}$, (4) $x=q$ and $y\in \widehat{N}_k$, or (5) $x=y$. We note that $p$ has no upper bound other than itself, and that this is also true for elements of $\widehat{N}_n$ for all $0\leq n\leq k$.
\end{proof}

\begin{corollary}
$P_k$ is a poset for all $k\in\omega$.
\end{corollary}
\begin{proof}
Since reflexivity is automatic, it remains only to check antisymmetry and transitivity, and these follow from Lemma \ref{L:height}
\end{proof}

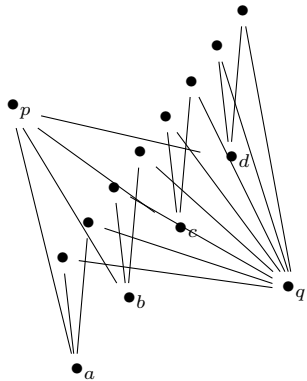
\begin{figure}[htbp]
\centering
\tdplotsetmaincoords{60}{110}
\begin{tikzpicture}[tdplot_main_coords]
\node (p) at (0, -2, 2) {$\phantom{x}\bullet_p$};
\node (d) at (-3,0,0) {$\phantom{x}\bullet_d$};
\node (c) at (-1,0,0) {$\phantom{x}\bullet_c$};
\node (b) at (1,0,0) {$\phantom{x}\bullet_b$};
\node (a) at (3,0,0) {$\phantom{x}\bullet_a$};
\node (d') at (-3.5,0,2) {$\bullet$};
\node (d'') at (-2.5,0,2) {$\bullet$};
\node (c') at (-1.5,0,2) {$\bullet$};
\node (c'') at (-0.5,0,2) {$\bullet$};
\node (b') at (0.5,0,2) {$\bullet$};
\node (b'') at (1.5,0,2) {$\bullet$};
\node (a') at (2.5,0,2) {$\bullet$};
\node (a'') at (3.5,0,2) {$\bullet$};
\node (q) at (0, 2, 0) {$\bullet_q$};

\draw (p)--(a);
\draw (p)--(b);
\draw (p)--(c);
\draw (p)--(d);

\draw (a)--(a');
\draw (b)--(b');
\draw (c)--(c');
\draw (d)--(d');
\draw (a)--(a'');
\draw (b)--(b'');
\draw (c)--(c'');
\draw (d)--(d'');

\draw (q)--(a');
\draw (q)--(b');
\draw (q)--(c');
\draw (q)--(d');
\draw (q)--(a'');
\draw (q)--(b'');
\draw (q)--(c'');
\draw (q)--(d'');
\end{tikzpicture}
\caption{The poset $P_0$}
\label{F:P0}
\end{figure}

\begin{figure}[htbp]
\centering
\tdplotsetmaincoords{60}{110}
\begin{tikzpicture}[tdplot_main_coords]
\node (p) at (0, -2, 2) {$\phantom{x}\bullet_p$};
\node (d) at (-3,0,0) {$\phantom{x}\bullet_d$};
\node (c) at (-1,0,0) {$\phantom{x}\bullet_c$};
\node (b) at (1,0,0) {$\phantom{x}\bullet_b$};
\node (a) at (3,0,0) {$\phantom{x}\bullet_a$};
\node (d') at (-3.5,0,2) {$\bullet$};
\node (d'') at (-2.5,0,2) {$\bullet$};
\node (c') at (-1.5,0,2) {$\bullet$};
\node (c'') at (-0.5,0,2) {$\bullet$};
\node (b') at (0.5,0,2) {$\bullet$};
\node (b'') at (1.5,0,2) {$\bullet$};
\node (a') at (2.5,0,2) {$\bullet$};
\node (a'') at (3.5,0,2) {$\bullet$};

\node (ab) at (6,2,0) {$\phantom{xx}\bullet_{e_{ab}}$};
\node (ac) at (4,2,0) {$\phantom{xx}\bullet_{e_{ac}}$};
\node (ad) at (2,2,0) {$\phantom{xx}\bullet_{e_{ad}}$};
\node (bc) at (0,2,0) {$\phantom{xx}\bullet_{e_{bc}}$};
\node (bd) at (-2,2,0) {$\phantom{xx}\bullet_{e_{bd}}$};
\node (cd) at (-4,2,0) {$\phantom{xx}\bullet_{e_{cd}}$};

\node (ab') at (5.5,2,2) {$\bullet$};
\node (ac') at (3.5,2,2) {$\bullet$};
\node (ad') at (1.5,2,2) {$\bullet$};
\node (bc') at (-0.5,2,2) {$\bullet$};
\node (bd') at (-2.5,2,2) {$\bullet$};
\node (cd') at (-4.5,2,2) {$\bullet$};

\node (ab'') at (6.5,2,2) {$\bullet$};
\node (ac'') at (4.5,2,2) {$\bullet$};
\node (ad'') at (2.5,2,2) {$\bullet$};
\node (bc'') at (0.5,2,2) {$\bullet$};
\node (bd'') at (-1.5,2,2) {$\bullet$};
\node (cd'') at (-3.5,2,2) {$\bullet$};

\node (q) at (1, 4, 0) {$\bullet_q$};

\draw (p)--(a);
\draw (p)--(b);
\draw (p)--(c);
\draw (p)--(d);

\draw (a)--(a');
\draw (b)--(b');
\draw (c)--(c');
\draw (d)--(d');
\draw (a)--(a'');
\draw (b)--(b'');
\draw (c)--(c'');
\draw (d)--(d'');

\draw (ab)--(ab');
\draw (ac)--(ac');
\draw (ad)--(ad');
\draw (bc)--(bc');
\draw (bd)--(bd');
\draw (cd)--(cd');
\draw (ab)--(ab'');
\draw (ac)--(ac'');
\draw (ad)--(ad'');
\draw (bc)--(bc'');
\draw (bd)--(bd'');
\draw (cd)--(cd'');

\draw (ab')--(q);
\draw (ac')--(q);
\draw (ad')--(q);
\draw (bc')--(q);
\draw (bd')--(q);
\draw (cd')--(q);
\draw (ab'')--(q);
\draw (ac'')--(q);
\draw (ad'')--(q);
\draw (bc'')--(q);
\draw (bd'')--(q);
\draw (cd'')--(q);

\draw (a')--(ab);
\draw (a'')--(ab);
\draw (b')--(ab);
\draw (b'')--(ab);

\end{tikzpicture}
\caption{The poset $P_1$. For the sake of visual clarity, with the exception of the lines between $a'$, $a''$, $b'$, $b''$ and $e_{ab}$, the ordering between $\widehat{N}_0$ and $N_1$ is not shown.}
\label{F:P1}
\end{figure}
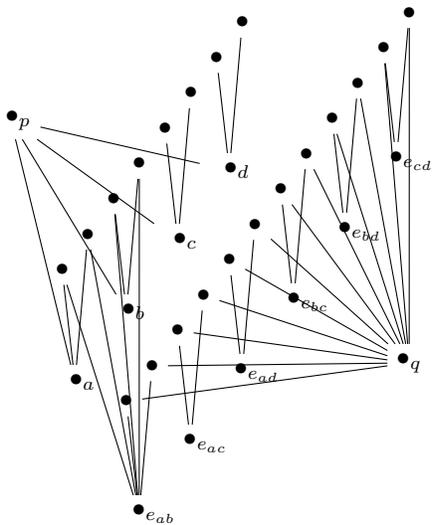

\begin{figure}[htbp]
\setlength{\unitlength}{1.4truemm}
\begin{picture}(150,30)
\put(-10,0){%
\begin{tikzpicture}[scale=.1]
\draw[fill] (-5,25) circle (.5);
\node at (-5.5,27) {$p$};
\draw[fill] (15,0) circle (.5);
\draw[fill] (20,0) circle (.5);
\draw[fill] (25,0) circle (.5);
\draw[fill] (30,0) circle (.5);
\node at (14.5,-2.5) {$a$};
\node at (19.5,-2) {$b$};
\node at (24.5,-2.5) {$c$};
\node at (29.5,-2) {$d$};
\draw[line width=.5pt] (-5,25) -- (15,0);
\draw[line width=.5pt] (-5,25) -- (20,0);
\draw[line width=.5pt] (-5,25) -- (25,0);
\draw[line width=.5pt] (-5,25) -- (30,0);
\draw[fill] (12,25) circle (.5);
\draw[fill] (14,25) circle (.5);
\draw[fill] (18,25) circle (.5);
\draw[fill] (20,25) circle (.5);
\draw[fill] (25,25) circle (.5);
\draw[fill] (27,25) circle (.5);
\draw[fill] (31,25) circle (.5);
\draw[fill] (33,25) circle (.5);
\node at (24,29) {$c'$};
\node at (27.5,29) {$c''$};
\node at (31,29) {$d'$};
\node at (35,29) {$d''$};
\draw[dotted, line width=1pt] (23.5,23.5)--(34.5,23.5)--(34.5,26.5)--(23.5,26.5)--cycle;
\draw[line width=.5pt] (12,25) -- (15,0) -- (14,25);
\draw[line width=.5pt] (18,25) -- (20,0) -- (20,25);
5
\draw[line width=.5pt] (25,25) -- (25,0) -- (27,25);
\draw[line width=.5pt] (31,25) -- (30,0) -- (33,25);
\draw[fill] (53,0) circle (.5);
\draw[fill] (51,25) circle (.5);
\draw[fill] (53,25) circle (.5);
\node at (53,-2.5) {$e_{cd}$};
\draw[line width=.5pt] (25,25) -- (53,0) -- (27,25);
\draw[line width=.5pt] (31,25) -- (53,0) -- (33,25);
\draw[line width=.5pt] (51,25) -- (53,0) -- (53,25);
\draw[fill] (150,0) circle (.5);
\node at (151,3) {$q$};
\draw[line width=.5pt] (12,-5) -- (33,-5) -- (33,3) -- (12,3) -- cycle;
\draw[line width=.5pt] (9,22) -- (38,22) -- (38,31) -- (9,31) -- cycle;
\draw[line width=.5pt] (50,-5) -- (80,-5) -- (80,3) -- (50,3) -- cycle;
\draw[line width=.5pt] (47,22) -- (83,22) -- (83,31) -- (47,31) -- cycle;
\draw[line width=.5pt] (110,-5) -- (140,-5) -- (140,3) -- (110,3) -- cycle;
\draw[line width=.5pt] (110,22) -- (140,22) -- (140,31) -- (110,31) -- cycle;
\draw[line width=.5pt] (110,22) -- (150,0) -- (140,31);
\node at (7,-1) {$N_0$};
\node at (5,27) {$\widehat{N}_0$};
\node at (45,-1) {$N_1$};
\node at (43,27) {$\widehat{N}_1$};
\node at (105,-1) {$N_k$};
\node at (106,27) {$\widehat{N}_k$};
\draw[fill] (90,15) circle (.7);
\draw[fill] (95,15) circle (.7);
\draw[fill] (100,15) circle (.7);
\end{tikzpicture}
}
\end{picture}
\bigskip
\bigskip
\caption{The poset $P_k$. This diagram was generously donated by the anonymous referee.}
\label{F:Pk}
\end{figure}
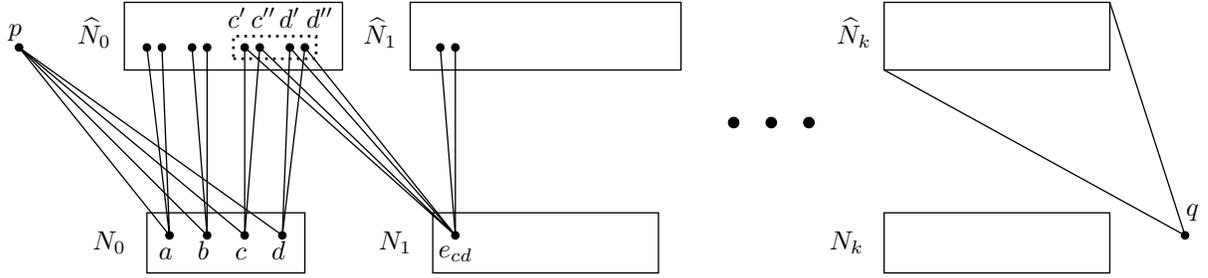

\begin{corollary}\label{C:minmax}
The maximal elements of $P_k$ are precisely the members of $\{p\}\cup\bigcup_{n=0}^k \widehat{N}_n$, and the minimal elements of $P_k$ are precisely the members of $\{q\}\cup\bigcup_{n=0}^k N_n$.
\end{corollary}
\begin{proof}
This is a restatement of the key observation in the proof of Lemma \ref{L:height}.
\end{proof}

\begin{corollary}\label{C:height}
The height of $P_k$ is 2 (here height is defined as being the length of the longest chain).
\end{corollary}
\begin{proof}
This follows from Lemma \ref{L:height} and the fact that $P_k$ contains chains of length $2$.
\end{proof}

\begin{lemma}\label{L:joins}
In $P_k$, the only non-trivial join is $p$, which is the join of every non-singleton subset of $N_0=\{a,b,c,d\}$. No other non-trivial joins are defined in $P_k$.
\end{lemma}
\begin{proof}
Let $y\neq p\in P_k$, and let $X\subset P_k$ be non-empty and such that $x<y$ for all $x\in X$. Then we must have $y\in \widehat{N}_n$ for some $0\leq n\leq k$, but then we cannot have $y=\bv X$ because, by construction, $y$ has an incomparable twin with the same lower bounds.
\end{proof}

\begin{corollary}
Every element of $P_k\setminus(\{p\}\cup N_0)$ is join-prime.
\end{corollary}

We obtain $P_{k+1}$ from $P_{k}$ by adding an extra block, $N_{k+1}\cup \widehat{N}_{k+1}$, between $N_k\cup\widehat{N}_k$ and $q$. Every $N_{k+1}$ contains ${|N_{k}| \choose 2}$ elements, so the structures grow rapidly.

\begin{lemma}\label{L:triples}
Let $k\geq 1$. Then:
\begin{enumerate}
\item If $\Gamma$ is a 3-filter of $P_k$, and $\Gamma$ contains at least 3 elements of $N_n$ (for $0\leq n < k$), then $\Gamma$ contains at least 3 members of $N_{n+1}$. 
\item Let $3\leq m,n \leq \omega$. Suppose $S$ is a three element subset of $N_n$, for some $n<k$. Then the smallest $(m,n)$-filter of $P_k$ containing $S$ contains exactly three elements of $N_{n+1}$.
\end{enumerate}
\end{lemma}
\begin{proof}
Suppose $\{x,y,z\}\subset \Gamma\cap N_n$. Then $\{x',x'',y',y'',z',z''\}\subset\Gamma$, by up-closure of $\Gamma$, and so $\{e_{xy},e_{xz},e_{yz}\}\subset\Gamma$ by closure under binary meets. For part 2, suppose $S =\{x,y,z\}\subset N_n$, and let $\Gamma$ be the $(m,n)$-filter generated by $S$. Then $\Gamma\cap \widehat{N}_n =\{x',x'',y',y'',z',z''\}$, and so $\Gamma\cap N_{n+1} = \{e_{xy}, e_{xz}, e_{yz}\}$.   
\end{proof}

\begin{proposition}\label{P:main}
Let $x,y\in P_k$ and suppose $x\not\leq y$. Then the following are equivalent: 
\begin{enumerate}
\item $x\in\{p\}\cup N_0$ and $y\in \{q\}\cup \widehat{N}_k$. 
\item There is no $3$-filter containing $x$ but not $y$.
\item There is no $\omega$-filter containing $x$ but not $y$.
\end{enumerate}
\end{proposition}
\begin{proof}
First we show $1.\implies 2.$ directly. If $\Gamma$ is a $3$-filter containing any one of $N_0=\{a,b,c,d\}$, then it must also contain $p$ by up-closure. So, by the $3$-primality $\Gamma$ it must also contain (at least) three members of $N_0$. So by Lemma \ref{L:triples}(1) it must contain at least 3 members of $N_k$, and thus, by up-closure and closure under binary meets, it must also contain $q$, and hence, by up-closure, also every element of $\widehat{N}_k$. 

That $2.\implies 3.$ is automatic, so we show $3.\implies 1.$ by proving the contrapositive. If $x\notin(\{p\}\cup N_0)$, then $x^\uparrow$ is an $\omega$-filter containing $x$ but not $y$, as $x$ is join-prime. If $x\in (\{p\}\cup N_0)$ but $y\notin \{q\}\cup \widehat{N}_k$, we can construct an $\omega$-filter $\Gamma$ containing $x$ but not $y$ by making a suitable choice for which element of $N_0$ is left out of $\Gamma$. To see this note that if we choose $z\in N_0$ and let $X=\{p\}\cup N_0\setminus\{z\}$ then there is a smallest $\omega$-filter containing $X$, $\Gamma_X$ say, generated deterministically by alternating closing upwards and closing under meets. It follows from Lemma \ref{L:triples}(2) that $\Gamma_X$ will contain exactly 3 elements of $N_n$ for all $0\leq n \leq k$. The key observation then is that, if $e=e_{uv}\in N_n$ for some $u,v\in N_{n-1}$, and either $u\notin \Gamma_X$ or $v\notin\Gamma_X$, then $e\notin\Gamma_X$.
\end{proof}

\begin{corollary}
$P_k$ is not 3-representable for all $k\in\omega$.
\end{corollary}
\begin{proof}
This is a trivial consequence of Proposition \ref{P:main}.
\end{proof}

Given $k\in \omega$, we can define a map $\iota_k:P_k\to P_{k+1}$. Here $P_k$ and $P_{k+1}$ are constructed as described at the beginning of this section. We assume that the carriers of $P_k$ and $P_{k+1}$ are disjoint, and we will distinguish elements of $P_{k+1}$ from their counterparts in $P_k$ by using an underline. So, $P_{k+1}$ is constructed recursively by starting with the base $\underline{N}_0=\{\underline{a},\underline{b},\underline{c},\underline{d}\}$, defining \[\underline{N}_{n+1} = \{\underline{e}_{\underline{x}\underline{y}}:\underline{x} \text{ and } \underline{y} \text{ are distinct elements of } \underline{N}_n\},\]   
and defining 
\[\underline{\widehat{N}}_n=\bigcup_{\underline{x}\in\underline{N}_n} \{\underline{x}', \underline{x}''\}.\]
The carrier of $P_{k+1}$ is then \[\{\underline{p},\underline{q}\}\cup \bigcup_{n=0}^{k+1} \underline{N}_n \cup \bigcup_{n=0}^{k+1} \underline{\widehat{N}}_n.\]
We define the order on $P_{k+1}$ in the obvious way, and can now define $\iota_k$ recursively as follows:
\begin{itemize}
\item $\iota_k(x) = 
\begin{cases}\underline{x} \text{ when } x\in N_0 \\ 
\underline{p} \text{ when } x=p \\
\underline{q} \text{ when } x=q 
\end{cases}$
\item Assuming $\iota_k$ has been defined on $N_n$ for $n< k$, let $x,y\in N_n$. We define $\iota_k$ on $N_{n+1}$ by $\iota_k(e_{xy}) = \underline{e}_{\iota_k(x)\iota_k(y)} = \underline{e}_{\underline{x}\underline{y}}$.   
\item Assuming $\iota_k$ has been defined on $N_n$ for all $n\leq k$, let $x\in N_m$ for $m\leq k$. We define $\iota_k$ on $\widehat{N}_m$ by $\iota_k(x')= \iota_k(x)'=\underline{x}'$, and $\iota_k(x'') = \iota_k(x)''=\underline{x}''$. 
\end{itemize}

So $\iota_k$ is essentially the inclusion of $P_k$ into $P_{k+1}$, modulo the fact that we require the carriers to be distinct. Now, given $k\leq l \in \omega$, we define the map $\iota_{kl}:P_k\to P_{l+1}$ to be the composition $\iota_l\circ\ldots \circ\iota_{k+1}\circ \iota_k$. The map $\iota_{kl}$ is almost an order embedding between $P_k$ and $P_{l+1}$, but fails to be because the order between $\widehat{N}_k$ and $q$ in $P_k$ does not translate into an order between $\iota_{kl}[\widehat{N}_k]$ and $\iota_{kl}(q)$ in $P_{l+1}$. 

Consider now the ultraproduct $\UPPi$, where $U$ is some non-principal ultrafilter over $\omega$. Given $k\in \omega$, and an element $x\in P_k$, we define $\bar{x}$ to be the sequence $(x,\iota_{kk}(x),\iota_{k(k+1)}(x),\ldots)\in\prod_{l=k}^\omega P_l$. Since $\bar{x}$ has terms in all $P_m$ where $m\geq k$, it follows that $\bar{x}$ defines an element $[\bar{x}]$ of $\UPPi$. To revisit the analogy between the maps $\iota_{kl}$ and inclusion functions, the map taking $x\in P_k$ to $[\bar{x}]$ can be thought of as an inclusion of $P_k$ into $\UPPi$. 

\begin{lemma}\label{L:up}
Let $x\in P_k\setminus\{q\}$, and let $[y]\in\UPPi$. Then $[\bar{x}]<[y]$ if and only if one of the following is true:
\begin{enumerate}
\item $x\in N_0$ and $[y]=[\bar{p}]$, 
\item $x\in N_n$ for some $n$, and $[y]\in\{[\bar{x}'],[\bar{x}'']\}$, or
\item $x=e_{uv}\in N_{n+1}$ for some $u,v\in N_n$, and $[y]\in\{[\bar{u}'],[\bar{u}''],[\bar{v}'],[\bar{v}'']\}$.
\end{enumerate} 
\end{lemma}
\begin{proof}
The `if' part is trivial, so we prove `only if'. Since $[\bar{x}]<[y]$ we must have $x\in N_n$ for some $n$ by corollaries \ref{C:minmax} and \ref{C:height}. This follows because $\bar{x}(i) = x$ on a large set for some $x$. We also have $y(i)>\bar{x}(i)$ on a large set, and so, by closure of ultrafilters under finite meets, we have $x < y(i)$ on a large set. 

Thus $x$ is a minimal element that is not $q$, and so must be in $N_n$ for some $n$. Now, if $x\in N_n$, then it has a finite set of upper bounds, and thus, by primality of ultrafilters, $[y]$ must be $[\bar{z}]$ for some $z> x$. So we must have either 1., 2. or 3. as required. 
\end{proof}

\begin{proposition}
$\UPPi$ is $\omega$-representable.
\end{proposition}
\begin{proof}
Let $[x],[y]\in \UPPi$ and suppose $[x]\not\leq[y]$. Suppose first that \[\{i\in \omega: x(i)\not\in\{p\}\cup N_0\text{ or }y(i)\not\in \{q\}\cup \widehat{N}_i\}\in U.\] Then, by Proposition \ref{P:main} and Lemma \ref{L:filters}, there is an $\omega$-filter of $\UPPi$ containing $[x]$ but not $[y]$. Suppose instead that 
\[\{i\in \omega: x(i)\in\{p\}\cup N_0\text{ and }y(i)\in \{q\}\cup \widehat{N}_i\}\in U.\]  
We define $\Gamma\subset\UPPi$ by $\Gamma=\bigcup_{k\in\omega}\{[\bar{z}]:z\in P_k\setminus\{{q\}}\}$. We claim that $\Gamma$ is an $\omega$-filter of $\UPPi$. That $\Gamma$ is up-closed and closed under existing finite meets follows from Lemma \ref{L:up}, and that its complement is closed under existing finite joins follows from the primality of ultrafilters and Lemma \ref{L:joins}.

Now, since $\{i\in \omega: x(i)\in\{p\}\cup N_0\}\in U$, we must have $[x]=[\bar{z}]$ for some $z\in \{p\}\cup N_0\subset P_0$, by primality of ultrafilters, and so $[x]\in\Gamma$. Since $\{i\in \omega:y(i)\in \{q\}\cup \widehat{N}_i\}\in U$, we must have $[y]\neq[\bar{z}]$ for all $z\in P_k\setminus\{q\}$, for all $k\in\omega$, and so $[y]\not\in\Gamma$. The result then follows from Theorem \ref{T:rep}. 
\end{proof}

\begin{theorem}
For all $m,n$ with $3\leq m,n \leq \omega$, the class of $(m,n)$-representable posets is not finitely axiomatizable.
\end{theorem}
\begin{proof}
We have shown that the complement of the class is not closed under ultraproducts, and thus cannot be elementary, by \L o\'s' theorem. Hence the class of $(m,n)$-representable posets cannot be finitely axiomatized.
\end{proof}

We note that the faint possibility remains that $(m,n)$-representability is finitely axiomatizable over the class of \emph{finite} posets. An axiomatization using only a finite number of variables would also be sufficient for a polynomial time decision algorithm \cite[Proposition 3.1]{Var95}. So the study of first-order axioms for classes of $(m,n)$-representable posets remains somewhat relevant to the $\mathbf{P}$ vs. $\mathbf{NP}$ question, though most reasonable people would presumably take this connection as powerful evidence that such axiomatizations do not exist.

\bibliographystyle{spmpsci}

\end{document}